\documentclass[a4paper,12pt]{amsart}
\usepackage{amsfonts}
\usepackage{amsmath}
\usepackage{amssymb}
\usepackage[a4paper]{geometry}
\usepackage{mathrsfs}
\usepackage[colorlinks]{hyperref}
\renewcommand\eqref[1]{(\ref{#1})} 
\numberwithin{equation}{section}
\theoremstyle{plain}
\newtheorem{thm}{Theorem}[section]
\newtheorem{prop}[thm]{Proposition}

\theoremstyle{definition}

\begin{document}
\title[On geometric estimates for pull-in voltage]
{On geometric estimates for some problems arising from modeling pull-in voltage in MEMS}

\author[Durvudkhan Suragan]{Durvudkhan Suragan}
\address{
  Durvudkhan Suragan:
   \endgraf
	Department of Mathematics
\endgraf
School of Science and Technology, Nazarbayev University
\endgraf
53 Kabanbay Batyr Ave, Astana 010000
\endgraf
Kazakhstan
\endgraf
  {\it E-mail address} {\rm durvudkhan.suragan@nu.edu.kz}}

\author[Dongming Wei]{Dongming Wei}
\address{
	Dongming Wei:
	\endgraf
	Department of Mathematics
	\endgraf
	School of Science and Technology, Nazarbayev University
	\endgraf
	53 Kabanbay Batyr Ave, Astana 010000
	\endgraf
	Kazakhstan
	\endgraf
	{\it E-mail address} {\rm dongming.wei@nu.edu.kz}
}
\thanks{The first author was partially supported by NU SPG. 
	No new data was collected or generated during the course of this research.}

 \keywords{$p$-MEMS problem, pull-in voltage, Talenti's comparison principle,
 geometric estimate}
 \subjclass{47J10, 35J60}

\begin{abstract}
In this paper for all $p>1$ we prove that the pull-in voltage of the $p$-MEMS (micro-electro mechanical systems) problems on a smooth bounded domain of $\mathbb R^{d}, d\geq1,$ is minimized by symmetrizing the domain and the permittivity profile. The proofs rely on some suitable version of Talenti's comparison principle.  We also  demonstrate our method to the multidimensional MEMS type problems on the whole space $\mathbb R^{d}, d\geq3,$ and the Dirichlet boundary value problems of second order uniformly elliptic differential operators.
\end{abstract}
     \maketitle
\section{Introduction}
\label{intro}

Let us recall the second order differential equation with the singular nonlinearity modeling stationary MEMS (micro-electro mechanical systems): 
\begin{equation}\label{MEMSLaplacian-Dirichlet}\left\{\begin{array}{l}
-\Delta u(x)=\lambda \frac{f(x)}{(1-u(x))^{2}},\; 0\leq u(x) < 1,\,x\in\Omega\subset \mathbb R^{d}, d\ge 1 \\
u(x)=0,\,\,\,x\in\partial\Omega.
\end{array}\right.\end{equation}
Here $f$ describes the varying permittivity profile of the elastic membrane with \\ $f\in C^{\alpha}(\bar{\Omega})$ for some $\alpha \in (0,1], \, 0\leq f \leq 1,$ and $f\not\equiv0$.
The Dirichlet pull-in voltage is defined as
$$\lambda^{\ast}(\Omega,f)=\sup\{\lambda>0 \;|\; (1.1) \,\text{possesses at least one classical solution}\}.$$

For the Dirichlet pull-in voltage the following inequality holds:
\begin{thm}[Proposition 2.2.1 \cite{EGG}] \label{prop:1} A ball $B$ (with the symmetrized permittivity profile) is a minimizer of the Dirichlet pull-in voltage  among all domains of given volume, i.e.,
	$$
	\lambda^{\ast}(B,f^{\ast})\leq \lambda^{\ast}(\Omega,f)
	$$
	for an arbitrary domain $\Omega\subset \mathbb R^{d}$ with $|\Omega|=|B|$, where $|\cdot|$ is the Lebesgue measure in $\mathbb R^{d},d\ge 1$. Here $f^{\ast}$ is the symmetric decreasing rearrangement of $f$. 
\end{thm}

In this paper we consider a generalized stationary MEMS problem, so called $p$-MEMS equation (see \cite{CES08}): 

\begin{equation}\label{intropMEMSLaplacian-Dirichlet}\left\{\begin{array}{l}
-\Delta_p u(x)=-\text{div}(|\nabla u|^{p-2}\nabla u)=\lambda f(x)g(u),\; 0\leq u(x) < 1,\,x\in\Omega\subset \mathbb R^{d}, \\
u(x)=0,\,\,\,x\in\partial\Omega.
\end{array}\right.\end{equation}
where $\lambda >0$ and 
 $f$ is a smooth positive function. The non-linearity $g(u)$ is a non-decreasing, positive function defined on $[0,1)$
with a singularity at $u=1$:
\begin{equation}
\underset{u\rightarrow 1^{-}}{\lim}  g(u)=+\infty.
\end{equation}
We also assume that $(\lambda f(x)g(u))^{\ast}=\lambda f^{\ast}(x)g(u^{\ast})$, where by $\ast$ we denote the symmetric decreasing rearrangement. Clearly, these assumptions enables us to apply Talenti's comparison principle, and to simplify the proofs. 
For example, $g(u)=(1-u)^{-m},\;m\in \mathbb{N},\; f=const>0,$ 
satisfies the all above assumptions. Note that in the case $d=2$, \eqref{intropMEMSLaplacian-Dirichlet} is applicable for some nonlinear material pull-in applications. This is motivated by considering the membrane equation for nonlinear materials with constitutive stress-strain equation in the form $\sigma=|\epsilon|^{p-2}\epsilon$ subject to pull-in, where 
$\sigma$ is the stress and $\epsilon$ is the strain.
Similarly, the $p$-MEMS pull-in voltage is defined as
$$\lambda^{\ast}_{p}(\Omega,f)=\sup\{\lambda>0 \;|\; (1.2) \,\text{possesses at least one classical solution}\}.$$

In \cite{CES} Castorina, Esposito and Sciunzi proved that $\lambda^{\ast}_{p}<\infty$ and for every $\lambda \in (0,\lambda^{\ast}_{p})$ there is a minimal (and semi-stable) solution $u_{\lambda}$ (i.e. $u_{\lambda}$ is the smallest positive solution of \eqref{intropMEMSLaplacian-Dirichlet} in a pointwise sense). 

The main result of the present paper is: A ball $B$ (with the symmetrized permittivity profile) is a minimizer of the Dirichlet $p$-MEMS pull-in voltage  among all domains of given volume, i.e.
$$
\lambda^{\ast}_{p}(B,f^{\ast})\leq \lambda^{\ast}_{p}(\Omega,f), \; 1<p<\infty,
$$
for an arbitrary smooth domain $\Omega\subset \mathbb R^{d}, d \ge 1,$ with $|\Omega|=|B|,$ where $|\cdot|$ is the Lebesgue measure in $\mathbb R^{d}$. Here $f^{\ast}$ is the symmetric decreasing rearrangement of $f$. 
To the best of our knowledge, the result seems new even for the case $p=2$ since we have the general singular nonlinearity $g(u)$.
Also, we consider a similar stationary MEMS problem, but in infinity domain, that is, on the whole $\mathbb R^{d}$:
\begin{equation}\label{MEMSNewton_intro}\left\{\begin{array}{l}
-\Delta u(x)= \frac{\lambda\,f}{(1-u(x))^{2}},\quad  0\leq u(x) < 1,\quad x\in\mathbb R^{d},\, d\geq 3, 
\\
u(x)\longrightarrow 0,\quad |x|\longrightarrow \infty,
\end{array}\right.\end{equation}
where $\lambda >0$ and $f=1$ in $\Omega$ with
$ \text{supp}\,f\subset \Omega \subset \mathbb{R}^{d}.$

To analyse the main difference of the problems \eqref{MEMSLaplacian-Dirichlet} and \eqref{MEMSNewton_intro} let us briefly discuss linear analogues of these problems. A linear analogue of Theorem \ref{prop:1} is so called a Rayleigh-Faber-Krahn inequality.
To recall it let us consider the minimization problem of the first eigenvalue of the Laplacian with
the Dirichlet boundary condition (among domains of given volume):
\begin{equation}\label{Laplacian-Dirichlet}\left\{\begin{array}{l}
-\Delta u(x)=\lambda^{D} u(x),\; x\in\Omega\subset \mathbb R^{d}, \\
u(x)=0,\,\,\,x\in\partial\Omega.
\end{array}\right.\end{equation}
The famous Rayleigh-Faber-Krahn inequality asserts that 
$$\lambda^{D}_{1}(B) \leq\lambda^{D}_{1}(\Omega),$$
for any $\Omega$ with $|\Omega| = |B|$, where $B\subset\mathbb R^{d}$ is a ball and  $|\cdot|$ is the Lebesgue measure in $\mathbb R^{d}$.
Note that an analogue of the Rayleigh-Faber-Krahn inequality for  general convolution type integral operators were given in \cite{Ruzhansky-Suragan:UMN} (see also \cite{Ruzhansky-Suragan:JMAA}).

Similarly, we can consider a linear version of the problem \eqref{MEMSNewton_intro}:
\begin{equation}\label{15}
-\Delta u(x)=\mu u(x), \,\,\,\ x\in\Omega\subset\mathbb{R}^{d},
\end{equation}
with the nonlocal integral boundary condition
\begin{equation}
-\frac{1}{2}u(x)+\int_{\partial\Omega}\frac{\partial\varepsilon_{d}(x-y)}{\partial n_{y}}u(y)d S_{y}-
\int_{\partial\Omega}\varepsilon_{d}(x-y)\frac{\partial u(y)}{\partial n_{y}}d S_{y}=0,\,\,
x\in\partial\Omega,
\label{16}
\end{equation}
where $\varepsilon_{d}$ is the fundamental solution of the Laplacian and  $\frac{\partial}{\partial n_{y}}$ denotes the outer normal
derivative at a point $y$ on the boundary $\partial\Omega$.
The spectral problem \eqref{15}-\eqref{16} is equivalent (see \cite{KS09}) to 
\begin{equation}\label{Newpot}
u(x)=\mu \int_{\Omega}\varepsilon_{d}(x-y)u(y)dy, \; x\in\Omega\subset\mathbb{R}^{d}.
\end{equation}
This also means  
\begin{prop}\cite{KS09}\label{prop_intro}
	The problem \eqref{MEMSNewton_intro} is equivalent to the nonlinear integral problem 
	\begin{equation}\label{nonNewpot}
	u(x)=\lambda\int_{\Omega}\varepsilon_{d}(x-y)\frac{1}{(1-u(y))^{2}}dy, \,\, 0\leq u(x)< 1.
	\end{equation}
\end{prop}
We refer \cite{Ruzhansky-Suragan:Newton} for further discussions and for spectral theory of \eqref{Newpot}. 

As another consequence of our method, in this paper we present similar geometric estimate for the pull-in voltage (upper bound of the spectrum) of the (nonlinear) Dirichlet boundary value problem for a second order uniformly elliptic differential operator
$$L=\sum_{j,k=1}^{d}\frac{\partial}{\partial x_{j}}\left(
a_{jk}(x)\frac{\partial}{\partial x_{k}} \right)$$
with $a_{jk}(x)=a_{kj}(x)$.

In Section \ref{SEC:prel} we briefly discuss some preliminary results, in particular, we recall the celebrated Talenti comparison principle \cite{Tal}, which states that the symmetric decreasing rearrangement (Schwarz rearrangement) of the Newtonian potential of a charge distribution is pointwise smaller than the potential resulting from symmetrizing the charge distribution itself. Talenti's comparison principle can be also extended to the Dirichlet $p$-Laplacian and the Dirichlet uniformly elliptic boundary value problems. Main results of this paper and their proofs will be given in Section \ref{SEC:main}. Talenti's comparison principle plays a key role in the proofs.

\section{Preliminaries}
\label{SEC:prel}

Let $\Omega$ be a measurable bounded domain of $\mathbb{R}^{d}$. An open ball (with origin $0$) $\Omega^{\ast}$ is called a symmetric rearrangement of $\Omega$
if $|B|=|\Omega|$ and
$$
\Omega^{\ast}=B=\left\{ x\in \mathbb{R}^{d}\mid \sigma_{d}|x|^{d}<
|\Omega|\right\},
$$
where $\sigma_{d}=\frac{2\pi^{\frac{d}{2}}}{\Gamma(\frac{d}{2})}$ is the surface area of the unit ball in $\mathbb{R}^{d}$.
Let $u$ be a nonnegative measurable function vanishing at infinity in the sense that all of its positive level sets have a finite measure,
i.e.,
$$
{\rm Vol}\left(\left\{ x|u(x)>t\right\} \right)<\infty,\quad  \forall t>0.
$$
To define a symmetric decreasing rearrangement of $u$ one uses (see, for example \cite{LL}) the layer-cake decomposition, which expresses a nonnegative function $u$ in terms of its level sets in the following way

$$ u(x)=\int_{0}^{\infty}\chi_{\left\{ u(x)>t\right\}}dt, 
$$
where $\chi$ is the characteristic function.
{\it{ Let $u$ be a nonnegative measurable function vanishing at infinity. Then 
		 \begin{equation}\label{1}
		u^{\ast}(x)=\int_{0}^{\infty}\chi_{\left\{ u(x)>t\right\}^{\ast}
		}dt
		 \end{equation}
		is called a symmetric decreasing rearrangement of the function $u$.}}
Note that the symmetric decreasing rearrangement is also sometimes called the Schwarz rearrangement. The simple definition \eqref{1} can be useful in many proofs, for example, if $0\leq v(x)-u(x),\,\, \forall x\in \mathbb{R}^{d}$, then we directly get
$$
u^{\ast}(x)=\int_{0}^{\infty}\chi_{\left\{u(x)>t\right\}^{\ast}
}dt\leq \int_{0}^{\infty}\chi_{\left\{v(x)>t\right\}^{\ast}
}dt=v^{\ast}(x),\,\,\,\,\,\, \forall x\in \mathbb{R}^{d}.
$$
That is, 	if
$$0\leq u(x)\leq v(x),\,\,\,\,\,\, \forall x\in \mathbb{R}^{d},$$
then
$$0\leq v^{\ast}(x)-u^{\ast}(x) ,\,\,\,\,\,\, \forall x\in \mathbb{R}^{d}.$$
Moreover, if $g$ is (nonnegative) increasing, then we have 
 \begin{equation}\label{gu=ug}
	g^{\ast}(u(x))=\int_{0}^{\infty}\chi_{\left\{ g(u(x))>t\right\}^{\ast}
	}dt=\int_{0}^{\infty}\chi_{\left\{g(u^{\ast}(x))>t\right\}
}dt=g(u^{\ast}(x)).
\end{equation}
 
\begin{thm}\label{talenti} [Talenti's comparison principle for the Laplacian ] Consider a (smooth) nonnegative function $f$ with $supp f\subset \Omega\subset \mathbb{R}^{d},\,d\geq3,$ for a bounded set $\Omega$, and its symmetric decreasing rearrangement $f^{\ast}$.
		If solutions $u$ and $v$ of
		$$-\Delta u=f,\quad -\Delta v=f^{\ast},$$
		vanish at infinity, then
		$$u^{\ast}(x)\leq v(x),\quad \forall x\in \mathbb{R}^{d}.$$
\end{thm} Note that $u$ and $v$ exist, and are uniquely determined by the equation, i.e.
$$u(x)=\int_{\Omega}\varepsilon_{d}(x-y)f(y)dy$$
and
$$v(x)=\int_{B}\varepsilon_{d}(x-y)f^{\ast}(y)dy,$$
where $\varepsilon_{d}(\cdot)$ is the fundamental solution of the Laplacian, that is,
$$\varepsilon_{d}(x-y)=\frac{1}{(d-2)\sigma_{d}|x-y|^{d-2}},\quad d\geq3,$$
and $\sigma_{d}$ is the surface area of $d$-dimensional unit ball.
They are nonnegative since the fundamental solution is nonnegative.
The inequality also holds for nonnegative measurable functions $f$ vanishing when $|x|\rightarrow \infty$.
It is also known that the fundamental solution $\sigma_{d},\,d\geq3,$
 does not change its formula under the symmetric decreasing rearrangement, see e.g.
 Lieb and Loss \cite{LL}.
Talenti's comparison principle can be extended to the Dirichlet boundary value problems of uniformly elliptic second order differential operators and $p$-Laplacian (see, e.g. Section 4.3 of \cite{Bur} as well as \cite{Tal} and \cite{Talenti'77}).

\begin{thm}\label{ptalenti} [Talenti's comparison principle for the Dirichlet $p$-Laplacian] Let $1<p<\infty$. Consider a (smooth) nonnegative function $f$ in a smooth bounded domain $\Omega\subset \mathbb R^{d},d\ge 1$, and its symmetric decreasing rearrangement $f^{\ast}$.
	Then solutions $u$ and $v$ of
	$$-\Delta_{p} u=f\;{\rm in}\,\Omega,\; u|_{\partial \Omega}=0,$$  
	and 
	$$-\Delta_{p} v=f^{\ast}\;{\rm in}\,B,\; v|_{\partial B}=0,$$
satisfy 
	$$u^{\ast}(x)\leq v(x),\quad \forall x\in B.$$
	Here $B$ is a ball centered at the origin with $|B|=|\Omega|,$ where $|\cdot|$ is the  Lebesgue measure in $\mathbb R^{d}$.
\end{thm}

\section{Main results}
\label{SEC:main}

\subsection{The pull-in voltage for the $p$-MEMS problem.} We consider the pull-in voltage for the $p$-MEMS problem with $1<p<\infty$:  
\begin{equation}\label{pMEMSLaplacian-Dirichlet}\left\{\begin{array}{l}
-\Delta_p u(x)=\lambda f(x)g(u),\; 0\leq u(x) < 1,\,x\in\Omega\subset \mathbb R^{d},d\ge 1, \\
u(x)=0,\,\,\,x\in\partial\Omega,
\end{array}\right.\end{equation}
where $\lambda >0$ (is the applied voltage) and 
(the permittivity profile) $f$ is a smooth positive function. The non-linearity $g(u)$ is a non-decreasing, positive function defined on $[0,1)$
with a singularity at $u=1$:
\begin{equation}
\underset{u\rightarrow 1^{-}}{\lim}  g(u)=+\infty.
\end{equation}
We also assume that $(\lambda f(x)g(u))^{\ast}=\lambda f^{\ast}(x)g(u^{\ast})$, where $\ast$  is for the symmetric decreasing rearrangement. For instance,  $fg(u)=\sigma(1-u)^{-m},\;m\in \mathbb{N},\;\sigma=const>0, $ 
satisfies the all above assumptions.  
As usual, the $p$-MEMS pull-in voltage is defined as
$$\lambda^{\ast}_{p}:=\lambda^{\ast}_{p}(\Omega,f)=\sup\{\lambda>0 \;|\; (3.1) \,\text{possesses at least one classical solution}\}.$$

The main result of this paper is 
\begin{thm}\label{pTHM_main} A ball $B$ (with $f^{\ast}$) is a minimizer of the Dirichlet $p$-MEMS pull-in voltage  among all domains of given volume, i.e.
	$$
	\lambda^{\ast}_{p}(B,f^{\ast})\leq \lambda^{\ast}_{p}(\Omega,f), \; 1<p<\infty,
	$$
	for an arbitrary smooth bounded domain $\Omega\subset \mathbb R^{d},\,d\geq 1,$ with $|\Omega|=|B|,$ where $|\cdot|$ is the  Lebesgue measure in $\mathbb R^{d}$.
\end{thm}

\begin{proof}[Proof of Theorem \ref{pTHM_main}]
	 In \cite{CES} it was proved that $\lambda^{\ast}_{p}<\infty$ and for every $\lambda \in (0,\lambda^{\ast}_{p})$ there is a minimal (and semi-stable) solution $u_{\lambda}$ (i.e. $u_{\lambda}$ is the smallest positive solution of \eqref{pMEMSLaplacian-Dirichlet} in a pointwise sense). Consider the following Picard iteration scheme   
	\begin{equation}\label{psequence}
	-\Delta_p u_{m}(x)=\lambda f(x)g(u_{m-1}),\;u_{0}(x)\equiv0,\;m=1,2,\ldots ,
	\end{equation}
	with $u_{m}=0$ on the boundary $\partial\Omega$ of the smooth bounded domain $\Omega$.  The sequence converges uniformly to a positive solution $u_{\lambda}$ satisfying  $u\geq u_{\lambda}$ in $\Omega$ (see \cite{CES}), where $u$ is a positive solution of \eqref{pMEMSLaplacian-Dirichlet}. 
	Consider the following two sequences 
	\begin{equation}\label{psequence1}
	-\Delta_p u_{n}(x)=\lambda f(x)g(u_{n-1})\quad{\rm in}\,\Omega,\;u_{n}(x)|_{\partial\Omega}=0, \;n=1,2,\ldots ,
	\end{equation}
	and 
	\begin{equation}\label{psequence2}
	-\Delta_p v_{n}(x)=\lambda f^{\ast}(|x|)g(v_{n-1})\quad{\rm in}\,B,\; v_{n}(x)|_{\partial B}=0, \;n=1,2,\ldots ,
	\end{equation}
	with $u_{0}\equiv0$ and $v_{0}\equiv0$. We have 
	$$-\Delta_{p} u_{1}(x)=\lambda f(x) g(0)\quad{\rm in}\,\Omega,\; u_{1}(x)|_{\partial\Omega}=0,$$
	and
	 $$-\Delta_{p} v_{1}(x)=\lambda f^{\ast}(|x|)g(0)\quad{\rm in}\,B,\; v_{1}(x)|_{\partial B}=0.$$
	Here $f^{\ast}$ is the symmetric decreasing rearrangement of the positive function $f$.
	Therefore, by Talenti's comparison principle for the Dirichlet $p$-Laplacian for $1<p<\infty$ (see Theorem \ref{ptalenti}) we obtain 
	\begin{equation} \label{pk=1}
	u_{1}^{\ast}(x)\leq v_{1}(x),\quad \forall x\in B,
	\end{equation}
	where $B$ is the ball centered at the origin with $|B|=|\Omega|.$
	We also have 
	\begin{equation}\label{psequence3}
	-\Delta_{p} u_{2}=\lambda f g(u_{1})\quad{\rm in}\,\Omega,\; u_{2}|_{\partial\Omega}=0.
	\end{equation}
	In addition, let us consider 
	\begin{equation*}	
		-\Delta_{p}\tilde{v}_{2}(x)=\lambda (f g(u_{1}))^{\ast}\quad{\rm in}\,B,\; \tilde{v}_{2}|_{\partial B}=0,
	\end{equation*}
	that is, 
	\begin{equation}
	\label{psequence4}
	-\Delta_{p}\tilde{v}_{2}(x)=\lambda f^{\ast} g(u^{\ast}_{1})\quad{\rm in}\,B,\; \tilde{v}_{2}|_{\partial B}=0.
	\end{equation}
	By Talenti's comparison principle for \eqref{psequence3} and \eqref{psequence4} we obtain 
	$$u_{2}^{\ast}(x)\leq \tilde{v}_{2}(x), \quad \forall x\in B.$$
	By using \eqref{pk=1} we get 
	\begin{equation*}	
	-\Delta_{p}\tilde{v}_{2}(x)=\lambda f^{\ast} g(u^{\ast}_{1})\leq \lambda f^{\ast} g(v_{1})=-\Delta_{p}v_{2}(x),
	\end{equation*}
	in $B$, that is, by the comparison principle for the operator $-\Delta_{p}$ (see, e.g. \cite{GS}) we get 
	$$\tilde{v}_{2}(x)\leq v_{2}(x), \quad \forall x\in B.$$ 
	This gives
	$$u_{2}^{\ast}(x)\leq v_{2}(x), \quad \forall x\in B.$$
	Further, by repeating this process we arrive at
	$u^{\ast}_{n}\leq v_{n}$ in $B$ for all $n\geq0,$  that is, $\underset{B}{\rm {max}}\, u^{\ast}_{n} \leq\underset{B}{\rm {max}}\, v_{n}$ for all $n\geq0.$
	Since $\underset{B}{\rm {max}}\, u^{\ast}_{n}= \underset{\Omega}{\rm {max}}\, u_{n}$, it means that 
	for a given $\lambda$ if $\{v_{n}\}$ converges, then the sequence $\{u_{n}\}$ is also convergent. This fact proves $
	\lambda^{\ast}_p(B,f^{\ast})\leq \lambda^{\ast}_{p}(\Omega, f)
	$.
\end{proof}

Since the singular nonlinearity $g$ belongs to a very general class Theorem \ref{pTHM_main} implies new results even for the case $p=2$. For example, it gives a new geometric pull-in voltage estimate for the electrostatic MEMS problem
with effects of Casimir force (see \cite{Lai}):
\begin{equation}\left\{\begin{array}{l}
-\Delta u(x)=\lambda \left(\frac{1}{(1-u(x))^{2}}+ \frac{\sigma}{(1-u(x))^{4}}\right),\; 0\leq u(x) < 1,\,x\in\Omega\subset \mathbb R^{d}, \\
u(x)=0,\,\,\,x\in\partial\Omega.
\end{array}\right.\end{equation}
Here $\sigma=const>0$, the second term on the right-hand side describes the Casimir force. It is easy to see that the assumptions of Theorem \ref{pTHM_main} are satisfied, since
the right hand side is an increasing positive function of $u$. 

Now we demonstrate our method to the multidimensional MEMS problems in the whole Euclidean space $\mathbb R^{d}$ and nonlinear Dirichlet boundary value problems of uniformly elliptic differential operators.

\subsection{The pull-in voltage for the Newtonian potential}

We consider the pull-in voltage for the stationary deflection of an infinity elastic membrane
satisfying  

\begin{equation}\label{MEMSNewton}\left\{\begin{array}{l}
-\Delta u(x)=\lambda \frac{f(x)}{(1-u(x))^{2}},\quad  0\leq u(x) < 1,\quad x\in\mathbb R^{d},\, d\geq 3, 
\\
u(x)\longrightarrow 0,\quad |x|\longrightarrow \infty,
\end{array}\right.\end{equation}
where $\lambda >0$ is the applied voltage and the permittivity profile $f$ is a constant with finite support, that is, $f=1$ in $\Omega$ with
$ \text{supp}\,f\subset \Omega \subset \mathbb{R}^{d}.$

As usual, the pull-in voltage is defined as
$$\lambda^{\ast}(\Omega)=\sup\{\lambda>0 \;|\; (3.1) \,\text{possesses at least one classical solution}\}.$$

\begin{thm} \label{THM:1}
There exists a positive pull-in voltage $\lambda^{\ast}< \infty$ such that
\begin{itemize}
	\item[a)] For any $\lambda<\lambda^{\ast}$, there exists at least one solution of \eqref{MEMSNewton}.
	
	\item[b)] For any $\lambda>\lambda^{\ast}$, there is no solution of \eqref{MEMSNewton}.	
\end{itemize}
	\end{thm}
	
\begin{proof}[Proof of Theorem \ref{THM:1}]
By Proposition \ref{prop_intro}
the problem \eqref{MEMSNewton} is equivalent to the nonlinear integral problem 
\eqref{nonNewpot}. Thus, since \eqref{nonNewpot} has the trivial solution $u=0$ with $\lambda=0$, by the implicit function theorem \eqref{nonNewpot} has a solution.
In addition, since the fundamental solution $\varepsilon_{d}$ is positive, the integral on the right hand sight of \eqref{nonNewpot} is positive. This means
that $\lambda$ must be positive, that is, $0<\lambda<\lambda^{\ast}.$  Now we need to show that 
$\lambda^{\ast}< \infty.$ Let $0\leq u(x)< 1$ be a solution 
of \eqref{nonNewpot}. We also use the following known fact (see \cite{Ruzhansky-Suragan:Newton}):
The first eigenvalue $\mu_{1}$ of the spectral problem 
\begin{equation}\label{sp_for_NP}
\phi_{1}(x)=\mu_{1}\int_{\Omega}\varepsilon_{d}(x-y)\phi_{1}(y)dy
\end{equation}  
is simple and positive as well as the corresponding eigenfunction $\phi_{1}$ can be chosen positive. Thus, let us multiply  \eqref{nonNewpot} by $\phi_{1}$ and integrate over $\Omega$, then we have 

	\begin{equation*}
\int_{\Omega}u(x)\phi_{1}(x)dx=
\lambda\int_{\Omega}\int_{\Omega}\varepsilon_{d}(x-y)\frac{1}{(1-u(y))^{2}}dy\phi_{1}(x)dx, \,\, 0\leq u(x)< 1.
\end{equation*}
 By \eqref{sp_for_NP} we obtain 
 	\begin{equation*}
 \int_{\Omega}u(x)\phi_{1}(x)dx=
 \frac{\lambda}{\mu_{1}}\int_{\Omega}\frac{\phi_{1}(y)}{(1-u(y))^{2}}dy, \,\, 0\leq u(x)< 1,
 \end{equation*}
 that is, 
 	\begin{equation*}
 \lambda=
 \frac{\mu_{1}\int_{\Omega}u(x)\phi_{1}(x)dx}{\int_{\Omega}\frac{\phi_{1}(y)}{(1-u(y))^{2}}dy}\leq  \frac{\mu_{1}\int_{\Omega}\phi_{1}(x)dx}{\int_{\Omega}\phi_{1}(y)dy}.
 \end{equation*}
 This means 
 	\begin{equation}\label{estimate}
 \lambda^{\ast}\leq  \frac{\mu_{1}\int_{\Omega}\phi_{1}(x)dx}{\int_{\Omega}\phi_{1}(y)dy}< \infty,
 \end{equation}
and there is no solution 
 of \eqref{nonNewpot} for any $\lambda>\lambda^{\ast}.$
 By the definition of $\lambda^{\ast}$ for any $\lambda\in (0,\lambda^{\ast})$ there exists 
 $\tilde{\lambda}\in (\lambda,\lambda^{\ast})$ for which 
 \eqref{nonNewpot} has a solution $u_{\tilde{\lambda}}$, that is, 
 	\begin{equation}\label{nonNewpot_tilde}
 u_{\tilde{\lambda}}(x)=\tilde{\lambda}\int_{\Omega}\varepsilon_{d}(x-y)
 \frac{1}{(1-u_{\tilde{\lambda}}(y))^{2}}dy\geq \lambda\int_{\Omega}\varepsilon_{d}(x-y)
 \frac{1}{(1-u_{\tilde{\lambda}}(y))^{2}}dy,
 \end{equation}
 
 This also means that $u_{\tilde{\lambda}}$ is a supsolution of 
 \eqref{nonNewpot_tilde} for the parameter $\lambda$.   On the other hand, 
 since 
 	\begin{equation} 0\leq \lambda\int_{\Omega}\varepsilon_{d}(x-y)dy,
 \end{equation}
 $u\equiv 0$ is a subsolution of 
 \begin{equation} u_{\tilde{\lambda}}(x)\leq \lambda\int_{\Omega}\varepsilon_{d}(x-y)
 \frac{1}{(1-u_{\tilde{\lambda}}(y))^{2}}dy.
\end{equation}
Therefore, by the method of sub- and supsolutions (see the proof of \cite[Theorem 2.1.1]{EGG}) we prove existence of a solution $u_{\lambda}$ of \eqref{nonNewpot} for any $\lambda\in (0,\lambda^{\ast})$.
\end{proof}
Now we are ready to prove the following result. 
	\begin{thm}\label{THM_main} We have 
	$$
	\lambda^{\ast}(f^{\ast})\leq \lambda^{\ast}(f)
	$$
	for the constant permittivity profile $f=1$ in a smooth bounded domain $\Omega$ satisfying the assumption 
	$ \text{supp}\,f\subset \Omega \subset \mathbb{R}^{d},\, d\geq3.$ 
\end{thm}

\begin{proof}[Proof of Theorem \ref{THM_main}]
	Let $u$ be any positive solution of \eqref{nonNewpot}. Define the sequence (the Picard iteration scheme)
		\begin{equation}\label{sequence}
	u_{m}(x)=\lambda\int_{\mathbb{R}^{d}}\varepsilon_{d}(x-y)\frac{f(y)}{(1-u_{m-1}(y))^{2}}dy,\; u_{0}(x)\equiv0,\;m=1,2,\ldots,
	\end{equation}
	with $f=1$ and $\text{supp}\,f\subset \Omega.$ We have $u>u_{0}\equiv0$ and whenever $u\geq u_{m-1}$, then 
	
		\begin{equation*}
u(x)-u_{m}(x)=\lambda\int_{\mathbb{R}^{d}}\varepsilon_{d}(x-y)f(y)
\left(\frac{1}{(1-u(y))^{2}}-\frac{1}{(1-u_{m-1}(y))^{2}}\right)dy\geq 0,
\end{equation*} 
for all $x\in\mathbb{R}^{d},$ that is,
$1>u\geq u_{m}$ in $\mathbb{R}^{d}$ for each $m\geq0.$ Moreover, from \eqref{sequence}
it is straightforward to see that the sequence  $\{u_{m}\}$ is monotone increasing. Thus, it converges uniformly to a positive solution $u_{\lambda}$ satisfying  $u\geq u_{\lambda}$ in $\mathbb{R}^{d}$. 
Consider the following two sequences 
	\begin{equation}\label{sequence1}
u_{n}(x)=\lambda\int_{\mathbb{R}^{d}}\varepsilon_{d}(x-y)\frac{f(y)}{(1-u_{n-1}(y))^{2}}dy,\; u_{0}(x)\equiv0,\;n=1,2,\ldots,
\end{equation}
with $ \text{supp}\,f\subset \Omega,$
and 
 	\begin{equation}\label{sequence2}
 v_{n}(x)=\lambda\int_{\mathbb{R}^{d}}\varepsilon_{d}(x-y)\frac{f^{\ast}(|y|)}{(1-v_{n-1}(y))^{2}}dy,\; v_{0}(x)\equiv0,\;n=1,2,\ldots.
 \end{equation}
We have 
$$-\Delta u_{1}=\lambda f,\quad -\Delta v_{1}=\lambda f^{\ast},$$
therefore, by Talenti's comparison principle for the Laplacian (see Theorem \ref{talenti}) we obtain 
	\begin{equation} \label{k=1}
u_{1}^{\ast}(x)\leq v_{1}(x),\quad \forall x\in \mathbb{R}^{d}.
 \end{equation}
We also have 
	\begin{equation}\label{sequence3}
-\Delta u_{2}=\lambda \frac{f}{(1-u_{1})^{2}}.
\end{equation}
In addition, let us consider 
	\begin{equation*}	
\tilde{v}_{2}(x)=\lambda\int_{\mathbb{R}^{d}}\varepsilon_{d}(x-y)\frac{f^{\ast}(|y|)}{(1-u^{\ast}_{1}(y))^{2}}dy,
\end{equation*}
that is, 
	\begin{equation}
	\label{sequence4}
	 -\Delta \tilde{v}_{2}=\lambda 
\frac{f^{\ast}}{(1-u^{\ast}_{1})^{2}}.
\end{equation}

Thus, by Theorem \ref{talenti} for \eqref{sequence3} and \eqref{sequence4} we obtain 
$$u_{2}^{\ast}(x)\leq \tilde{v}_{2}(x), \quad \forall x\in \mathbb{R}^{d}.$$
By using \eqref{k=1} we get 
	\begin{equation*}	
\tilde{v}_{2}(x)=\lambda\int_{\mathbb{R}^{d}}
\varepsilon_{d}(x-y)\frac{f^{\ast}(|y|)}
{(1-u^{\ast}_{1}(y))^{2}}dy\leq \lambda\int_{\mathbb{R}^{d}}
\varepsilon_{d}(x-y)\frac{f^{\ast}(|y|)}
{(1-v_{1}(y))^{2}}dy=v_{2}(x),
\end{equation*}
in $\mathbb{R}^{d}$, that is,
 $$u_{2}^{\ast}(x)\leq v_{2}(x), \quad \forall x\in \mathbb{R}^{d}.$$

 Further, by continuing this process we obtain that
  $u^{\ast}_{n}\leq v_{n}$ for all $n\geq0,$  that is, $\underset{B}{\rm {max}}\, u^{\ast}_{n} \leq\underset{B}{\rm {max}}\, v_{n}$ for all $n\geq0.$
  Since $\underset{B}{\rm {max}}\, u^{\ast}_{n}= \underset{\Omega}{\rm {max}}\, u_{n}$, it means that 
  for a given $\lambda$ if $\{v_{n}\}$ converges, then the sequence $\{u_{n}\}$ is also convergent. Thus, we arrive at $
  \lambda^{\ast}(f^{\ast})\leq \lambda^{\ast}(f)
  $.
\end{proof}

We have the following upper bound for the pull-in voltage (for the non-constant permittivity profile):
\begin{prop} Let $f$ be an integrable function with
		$ \text{supp}\,f\subset \Omega \subset \mathbb{R}^{d},\, d\geq3.$  Let 
	$\mu_{1}(\Omega)$ be the first eigenvalue of the Newtonian potential \eqref{sp_for_NP} in $\Omega$. Then 
	\begin{equation}\label{upperbounds}
	 \lambda^{\ast}(f)\leq \frac{4 \mu_{1}(\Omega)}{27} (\underset{\Omega}{\rm {inf}}\, f)^{-1}.
	\end{equation}
	
\end{prop}

\begin{proof}[Proof of Proposition \ref{upperbounds}]
As in the	proof of Theorem \ref{THM:1}, for any $\lambda\in (0,\lambda^{\ast})$ we have 
	\begin{equation*}
\int_{\Omega}u(x)\phi_{1}(x)dx=
\frac{\lambda}{\mu_{1}}\int_{\Omega}\frac{\phi_{1}(y)f(y)}{(1-u(y))^{2}}dy, \,\, 0\leq u< 1.
\end{equation*}
Since $u(1-u)^{2}\leq \frac{4}{27}$ we obtain
	\begin{equation*}
\int_{\Omega}u(x)\phi_{1}(x)dx=
\frac{\lambda}{\mu_{1}}\int_{\Omega}\frac{u(y)\phi_{1}(y)f(y)}{u(y)(1-u(y))^{2}}dy\geq \frac{27\lambda\, \underset{\Omega}{\rm {inf}}\, f}{4\mu_{1}}\int_{\Omega}u(y)\phi_{1}(y)dy
\end{equation*}
proving the inequality \eqref{upperbounds}.
	\end{proof}

Note that, moreover, one can prove other upper estimates of the pull-in voltage that depends on the global properties of  the (non-constant) permittivity profile. For instance, for the Dirichlet case (see, e.g. \cite{GPW}) we have the estimate	
\begin{equation}
(\lambda^D)^{\ast}(f)\leq \frac{4 \mu^{D}_{1}(\Omega)}{3} \frac{\int_{\Omega}\phi^{D}_{1}(x)dx}{\int_{\Omega}\phi^{D}_{1}(y)f(y)dy},
\end{equation}
where $\mu^{D}_{1}$ and $\phi^{D}_{1}$ are the first eigenvalue and the first eigenfunction of 
the Dirichlet Laplacian, respectively.   

\subsection{The pull-in voltage for uniformly elliptic problems.}

 We consider the pull-in voltage problem for the second order uniformly elliptic differential operator:  
\begin{equation}\label{uniformly-elliptic}\left\{\begin{array}{l}
-L u(x)=\lambda f(x)g(u),\;0<u,\; x\in\Omega\subset \mathbb R^{d}, \\
u(x)=0,\,\,\,x\in\partial\Omega.
\end{array}\right.\end{equation}
where 
$$L=\sum_{j,k=1}^{d}\frac{\partial}{\partial x_{j}}\left(
a_{jk}(x)\frac{\partial}{\partial x_{k}} \right)$$
with $a_{jk}(x)=a_{kj}(x)$ is a second order uniformly elliptic differential operator, that is,  there exists a positive constant $c$ such that
$$\sum_{j,k=1}^{d}a_{j,k}(x)\xi_{j}\xi_{k}\geq c\sum_{j}^{d}\xi_{j}^{2}
$$
for all $\xi$ and $x\in \Omega.$
Let us assume that $f$ and $g$ are (smooth) positive functions
such that $\frac{\partial g(t)}{\partial t}>0$ in $R^{
+}$ and $(f(x)g(u(x)))^{\ast}= f^{\ast}(x)g(u^{\ast}(x))$, 

As before, $\lambda >0$ is the applied voltage and the pull-in is defined as
$$\lambda^{\ast}_{L}:=\lambda^{\ast}_{L}(\Omega,f)=\sup\{\lambda>0 \;|\; (3.23) \,\text{possesses at least one classical solution}\}.$$
We obtain the following geometric estimate for $\lambda^{\ast}_{L}$.
\begin{thm}\label{uniformly-ellipticTHM_main} A ball $B$ (with the symmetrized permittivity profile) is a minimizer of the Dirichlet pull-in voltage of the operator $L$ among all domains of given volume, i.e.
	$$
	\lambda^{\ast}_{L}(B,f^{\ast})\leq \lambda^{\ast}_{L}(\Omega,f), 
	$$
	for an arbitrary smooth domain $\Omega\subset \mathbb R^{d}, \,d\geq 1,$ with $|\Omega|=|B|,$ where $|\cdot|$ is the  Lebesgue measure in $\mathbb R^{d}$.
\end{thm}

\begin{proof}[Proof of Theorem \ref{uniformly-ellipticTHM_main}]
	In \cite[Chapter IV]{Bandle} it was proved that $\lambda^{\ast}_{L}<\infty$ and for every $\lambda \in (0,\lambda^{\ast}_{L})$ there is a minimal (and semi-stable) solution $u_{\lambda}$ (i.e. $u_{\lambda}$ is the smallest positive solution of \eqref{uniformly-elliptic} in a pointwise sense). Define the sequence 
	\begin{equation}\label{unisequence}
	-L u_{m}(x)=\lambda f(x)g(u_{m-1}),\;u_{0}(x)\equiv0, \;m=1,2,\ldots,
	\end{equation}
	with $u_{m}=0$ on the boundary $\partial\Omega$ of the smooth bounded domain $\Omega$.  Dirichlet boundary value problem \eqref{uniformly-elliptic} is solvable if and only if the sequence $\{u_{m}\}$ is uniformly bounded. Moreover,  the sequence $\{u_{m}\}$ is uniformly bounded, then it  converges uniformly to a minimal solution $u_{\lambda}$ satisfying  $u\geq u_{\lambda}$ in $\Omega$ (see \cite{Keller-Cohen}), where $u$ is a positive solution of \eqref{uniformly-elliptic}. 
	As in the previuos proofs let us consider the following two sequences 
	\begin{equation}\label{unisequence1}
	-L u_{n}(x)=\lambda f(x)g(u_{n-1})\quad{\rm in}\,\Omega,\;u_{n}(x)|_{\partial\Omega}=0, \;n=1,2,\ldots,
	\end{equation}
	and 
	\begin{equation}\label{unisequence2}
	-L v_{n}(x)=\lambda f^{\ast}(|x|)g(v_{n-1})\quad{\rm in}\,B,\; v_{n}(x)|_{\partial B}=0, \;n=1,2,\ldots,
	\end{equation}
	with $u_{0}\equiv0$ and $v_{0}\equiv0$. We have 
	$$-L u_{1}(x)=\lambda f(x) g(0)\quad{\rm in}\,\Omega,\; u_{1}|_{\partial\Omega}=0,$$
	and
	$$-L v_{1}(x)=\lambda f^{\ast}(|x|)g(0)\quad{\rm in}\,B,\; v_{1}(x)|_{\partial B}=0.$$
	Here $f^{\ast}$ is the symmetric decreasing rearrangement of the positive function $f$.
	Therefore, by Talenti's comparison principle for the Dirichlet boundary value problem for second order uniformly elliptic differential operators  we obtain 
	\begin{equation} \label{unik=1}
	u_{1}^{\ast}(x)\leq v_{1}(x),\quad \forall x\in B,
	\end{equation}
	where $B$ is the ball centered at the origin with $|B|=|\Omega|.$
	We also have 
	\begin{equation}\label{unisequence3}
	-L u_{2}=\lambda f g(u_{1})\quad{\rm in}\,\Omega,\; u_{2}|_{\partial\Omega}=0.
	\end{equation}
	In addition, let us consider 
	\begin{equation*}	
	-L\tilde{v}_{2}(x)=\lambda (f g(u_{1}))^{\ast}\quad{\rm in}\,B,\; \tilde{v}_{2}|_{\partial B}=0,
	\end{equation*}
	that is, 
	\begin{equation}
	\label{unisequence4}
	-L\tilde{v}_{2}(x)=\lambda f^{\ast} g(u^{\ast}_{1})\quad{\rm in}\,B,\; \tilde{v}_{2}|_{\partial B}=0.
	\end{equation}
	By Talenti's comparison principle for \eqref{unisequence3} and \eqref{unisequence4} we obtain 
	$$u_{2}^{\ast}(x)\leq \tilde{v}_{2}(x), \quad \forall x\in B.$$
	By using \eqref{unik=1} we get 
	\begin{equation*}	
	-L\tilde{v}_{2}(x)=\lambda f^{\ast} g(u^{\ast}_{1})\leq \lambda f^{\ast} g(v_{1})=-L v_{2}(x),
	\end{equation*}
	in $B$, that is, by the (classical) comparison principle for $-L$ we get 
	$$\tilde{v}_{2}(x)\leq v_{2}(x), \quad \forall x\in B.$$ 
	This gives
	$$u_{2}^{\ast}(x)\leq v_{2}(x), \quad \forall x\in B.$$
	By induction we arrive at
	$u^{\ast}_{n}\leq v_{n}$ in $B$ for all $n\geq0,$  that is, $\underset{B}{\rm {max}}\, u^{\ast}_{n} \leq\underset{B}{\rm {max}}\, v_{n}$ for all $n\geq0.$
	Since $\underset{B}{\rm {max}}\, u^{\ast}_{n}= \underset{\Omega}{\rm {max}}\, u_{n}$, it means that 
	for a given $\lambda$ if $\{v_{n}\}$ converges, then the sequence $\{u_{n}\}$ is also convergent. This fact proves $
	\lambda^{\ast}_{L}(B,f^{\ast})\leq \lambda^{\ast}_{L}(\Omega, f)
	$.
\end{proof}

\medskip

{\bf Conflict of Interest Statement.}

The authors declare that there is no conflict of
interest.

\end{document}